\documentclass[12pt,leqno]{article}
\usepackage{amssymb, amscd, amsmath, amsthm}
\usepackage{dsfont}
\allowdisplaybreaks

\def\beq{\begin{equation}}
\def\eeq{\end{equation}}

\theoremstyle{definition}
\newtheorem{definition}{Definition}
\newtheorem{example}{Example}

\theoremstyle{plain}
\newtheorem{theorem}{Theorem}

\newtheorem{proposition}{Proposition}
\newtheorem{claim}{Claim}

\newtheorem{corollary}{Corollary}
\newtheorem{lemma}{Lemma}

\numberwithin{theorem}{section}
\numberwithin{definition}{section}
\numberwithin{claim}{section}
\numberwithin{lemma}{section}
\numberwithin{conjecture}{section}
\numberwithin{corollary}{section}
\numberwithin{equation}{section}
\numberwithin{example}{section}

\begin{document}

\title{Almost intersecting families}

\author{Peter Frankl\thanks{R\'enyi Institute, Budapest, Hungary and MIPT, Moscow.}
\ and Andrey Kupavskii\thanks{MIPT, Moscow, IAS, Princeton and CNRS, Grenoble}}

\date{}

\maketitle

\begin{abstract}
Let $n > k > 1$ be integers, $[n] = \{1, \ldots, n\}$.
Let $\mathcal F$ be a family of $k$-subsets of~$[n]$.
The family $\mathcal F$ is called \emph{intersecting} if $F \cap F' \neq \emptyset$ for all $F, F' \in \mathcal F$.
It is called \emph{almost intersecting} if it is \emph{not} intersecting but to every $F \in \mathcal F$ there is at most one $F'\in \mathcal F$ satisfying $F \cap F' = \emptyset$.
Gerbner et al.\ \cite{GLPPS} proved that if $n \geq 2k + 2$ then $|\mathcal F| \leq {n - 1\choose k - 1}$ holds for almost  intersecting families.
The main result (Theorem \ref{th:1.6}) implies the considerably stronger and best possible bound $|\mathcal F| \leq {n - 1\choose k - 1} - {n - k - 1\choose k - 1} + 2$ for $n > (2 + o(1))k$.
\end{abstract}

\section{Introduction}
\label{sec:1}

Let $[n] = \{1, \dots, n\}$ be the standard $n$-element set, $2^{[n]}$ its power set and ${[n]\choose k}$ the collection of all its $k$-subsets.
Subsets of $2^{[n]}$ are called \emph{families}.

A family $\mathcal F$ is called \emph{intersecting} if $F \cap G \neq \emptyset$ for all $F, G \in \mathcal F$.
One of the fundamental results in extremal set theory is the Erd\H{o}s--Ko--Rado Theorem:

\begin{theorem}[\cite{EKR}]
\label{th:1.1}
Suppose that $\mathcal F \subset{[n]\choose k}$ is intersecting, $n \geq 2k > 0$.
Then
\beq
\label{eq:1.1}
|\mathcal F| \leq {n - 1\choose k - 1}.
\eeq
\end{theorem}

Gerbner et al.\ \cite{GLPPS} proved an interesting generalisation of \eqref{eq:1.1}.
To state it we need a definition.

\setcounter{definition}{1}
\begin{definition}
\label{def:1.2}
A family $\mathcal F \subset 2^{[n]}$ is called \emph{almost intersecting} if it is \emph{not} intersecting, but to every $F \in \mathcal F$ there is at most one $G \in \mathcal F$ satisfying $F \cap G = \emptyset$.
\end{definition}

\setcounter{theorem}{2}
\begin{theorem}[\cite{GLPPS}]
\label{th:1.3}
Suppose that $n \geq 2k + 2$, $k \geq 1$, $\mathcal F \subset {[n]\choose k}$.
If $\mathcal F$ is intersecting or almost intersecting then \eqref{eq:1.1} holds.
\end{theorem}

A natural example of almost intersecting families is ${[2k]\choose k}$.
For $n = 2k$ and $2k + 1$ the best possible bound $|\mathcal F| \leq {2k\choose k}$ is proven in \cite{GLPPS}.

To present another example let us first define some $k$-uniform intersecting families.
For integers $1 \leq a \leq b \leq n$ set $[a,b] = \{a, a + 1, \ldots, b\}$.
For a fixed $x \in [n]$ let $\mathcal S = \mathcal S(n, k, x)$ be the full star with center in $x$, i.e., $\mathcal S = \left\{S \in {[n]\choose k}: x \in S\right\}$.
Every non-empty family $\mathcal F \subset \mathcal S$ for some $x$ is called a star.

For $3 \leq r \leq k + 1$ let us define
$$
\aligned
\mathcal B_r = \mathcal B_r(n,k) &= \left\{B \in {[n]\choose k} : 1 \in B, \, B \cap [2,r] \neq \emptyset\right\} \cup\\
&\quad \cup \left\{B \in {[n]\choose k} : 1 \notin B, [2,r] \subset B\right\}.
\endaligned
$$
Obviously, $|B_r| = {n - 1\choose k - 1} - {n - r\choose k - 1} + {n - r\choose k - r + 1}$.
In particular, $|\mathcal B_3| = |\mathcal B_4|$.
For $n > 2k$ one has
$$
|\mathcal B_4| < |\mathcal B_5| < \ldots < |\mathcal B_{k + 1}|.
$$
The family $\mathcal B_{k + 1}$ is called the Hilton--Milner family.
It has a single set, namely $[2, k + 1]$, which does not contain~$1$.

For $x,y \in [n]$ let us recall the standard notation:
\begin{align*}
\mathcal F(x) &= \{F \setminus \{x\} : x \in F \in \mathcal F\}, \mathcal F(\bar x) = \{F \in \mathcal F : x \notin F\},\\ \mathcal F(x, \bar y) &= \mathcal F(\bar y, x) = \{F \setminus \{x\}: x \in F \in \mathcal F, y \notin F\}.
\end{align*}
The \emph{maximum degree} $\Delta(\mathcal F)$ of a family $\mathcal F \subset 2^{[n]}$ is $\max\{|\mathcal F(x)| : x \in [n]\}$.
For $3 \leq r \leq k + 1$,
$$
\Delta(\mathcal B_r) = {n - 1\choose k - 1} - {n - r\choose k - 1} = {n - 2\choose k - 2} + \ldots + {n - r\choose k - 2} = |\mathcal B_r(1)|.
$$

Hilton and Milner \cite{HM} proved the following stability result for intersecting families. (This theorem has many proofs, see e.g. \cite{KZ}.)

\begin{theorem}[\cite{HM}]
\label{th:1.4}
Suppose that $n > 2k \geq 4$, $\mathcal F \subset {[n]\choose k}$ is intersecting, but $\mathcal F$ is not a star (not contained in a full star).
Then
\beq
\label{eq:1.2}
|\mathcal F| \leq |\mathcal B_{k + 1}|,
\eeq
moreover, equality holds only if $\mathcal F$ is isomorphic to $\mathcal B_{k + 1}$ or $k = 3$ and $\mathcal F$ is isomorphic to $\mathcal B_3$.
\end{theorem}

\setcounter{example}{4}
\begin{example}
\label{ex:1.5}
Let $B \subset {[n]\choose k}$ be an arbitrary set satisfying $1 \in B$, $B \cap [2, k + 1] = \emptyset$.
Set $\mathcal B^+ = \mathcal B_{k + 1} \cup \{B\}$.
Then $|\mathcal B^+| = |\mathcal B_{k + 1}| + 1$ and $\mathcal B^+$ is almost intersecting.
\end{example}

Our main result is the following.

\setcounter{theorem}{5}
\begin{theorem}
\label{th:1.6}
Suppose that $\mathcal F \subset {[n]\choose k}$ is almost intersecting, $k \geq 3$.
Then
\beq
\label{eq:1.3}
|\mathcal F| \leq |\mathcal B^+|={n-1\choose k-1}-{n-k-1\choose k-1}+2
\eeq
holds in the following cases:
\begin{itemize}
\item[{\rm (i)}] \ $k = 3$, $n \geq 13$,

\item[{\rm (ii)}] \ $k \geq 4$, $n \geq 3k + 3$,

\item[{\rm (iii)}] \ $k \geq 10$, $n > 2k + 2\sqrt{k} + 4$.
\end{itemize}
Moreover, equality in \eqref{eq:1.3} is only possible when $\mathcal F$ is isomorphic to $\mathcal B^+$. 
\end{theorem}
In what follows, we omit floor and ceiling signs whenever they do not affect the calculations.

The case $k = 2$ is easy.
Suppose that $\mathcal G \subset {[n]\choose 2}$ is almost intersecting and let $F, G \in \mathcal G$ be pairwise disjoint.
Set $X = F \cup G$ and note $|X| = 4$.

\setcounter{claim}{6}
\begin{claim}
\label{cl:1.7}
$\mathcal G \subset {X\choose 2}$.
\end{claim}

\begin{proof}
If $\mathcal G = \{F, G\}$ then we have nothing to prove.
On the other hand, for any further edge $H \in \mathcal G$, both $F \cap H$ and $G \cap H$ must be non-empty.
Since $|H| = 2$, $H \subset X$ follows.
\end{proof}
Note that the family ${[4]\choose 2}$ is the (unique, up to a permutation) extremal example in this case.

Let us make two simple but important observations.

\setcounter{proposition}{7}
\begin{proposition}
\label{pr:1.8}
Let $\mathcal F \subset {[n]\choose k}$ be almost intersecting.
Then there is a unique partition $\mathcal F = \mathcal F_0 \sqcup \mathcal P_1 \sqcup \ldots \sqcup \mathcal P_\ell$ where $\mathcal F_0$ is intersecting ($\mathcal F_0 = \emptyset$ is allowed) and for $1 \leq i \leq \ell$, $\mathcal P_i = \{P_i, Q_i\}$ with $P_i \cap Q_i = \emptyset$.
\end{proposition}

The above partition of $\mathcal F$ is called the \emph{canonical} partition.
The function $\ell(\mathcal F) = \ell$ is an important parameter of $\mathcal F$.

\setcounter{definition}{8}
\begin{definition}
\label{def:1.9}
A family $\mathcal T = \{T_1, \ldots, T_\ell\}$ satisfying $T_i \in \mathcal P_i$, is called a full tail (of $\mathcal F$).
\end{definition}

\setcounter{proposition}{9}
\begin{proposition}
\label{pr:1.10}
There are $2^\ell$ full tails $\mathcal T$ and for each of them $\mathcal F_0 \cup \mathcal T$ is intersecting.
\end{proposition}

Let us close this section by a short proof of \eqref{eq:1.3} for the special case $\ell(\mathcal F) = 1$.

There are two cases to consider according whether the families $\mathcal F_0 \cup \{P_1\}$, $\mathcal F_0\cup\{Q_1\}$ are stars or not.
Suppose first that one of them, say $\mathcal F_0 \cup \{P_1\}$ is not a star.
By Theorem \ref{th:1.4}, $\bigl|\mathcal F_0 \cup \{P_1\}\bigr| = |\mathcal F| - 1 \leq \bigl|\mathcal B_{k + 1}\bigr|$, implying \eqref{eq:1.3}.
For $k \geq 4$ uniqueness in the Hilton--Milner Theorem implies uniqueness in Theorem \ref{th:1.6} as well.
In the case $k = 3$, one has the extra possibility $\mathcal F_0 \cup \{P_1\} = \mathcal B_3$.
However, it is easy to check that adding a new $3$-set to $\mathcal B_3$ will \emph{never} produce an almost intersecting family.

The second case is even easier.
If both $\mathcal F_0 \cup \{P_1\}$ and $\mathcal F_0 \cup \{Q_1\}$ are stars then $P_1 \cap Q_1 = \emptyset$ implies that there are two distinct elements (the centres of the stars) $x, y$ such that $\{x, y\} \subset F$ for all $F \in \mathcal F_0$.
Consequently,
$$
|\mathcal F| = |\mathcal F_0| + 2 \leq {n - 2\choose k - 2} + 2 \leq {n - 2\choose k - 2} + 2{n - 3\choose k - 2} = |\mathcal B_3| \leq |\mathcal B_{k + 1}| < |\mathcal B^+|.
$$

\section{Preliminaries}
\label{sec:2}

Let us first prove an inequality on the size $\ell = \ell(\mathcal F)$ of full tails.

\begin{proposition}
\label{pr:2.1}
\beq
\label{eq:2.1}
\ell(\mathcal F) \leq {2k - 1\choose k - 1}.
\eeq
\end{proposition}

The proof of \eqref{eq:2.1} depends on a classical result of Bollob\'as \cite{B}.

\setcounter{theorem}{1}
\begin{theorem}[\cite{B}, cf.\ also \cite{JP} and \cite{Ka1}]
\label{th:2.2}
Suppose that $a, b$ are positive integers, $\mathcal A = \{A_1, \dots, A_m\}$, $\mathcal B = \{B_1, \dots, B_m\}$ are families satisfying $|A_i| = a$, $|B_i| = b$, $A_i \cap B_i = \emptyset$ for $1 \leq i \leq m$ and also
\beq
\label{eq:2.2}
A_i \cap B_j \neq \emptyset \ \ \ \text{ for all } \ \ 1 \leq i \neq j \leq m.
\eeq
Then
\beq
\label{eq:2.3}
m \leq {a + b\choose a}.
\eeq
\end{theorem}

\begin{proof}[Proof of Proposition \ref{pr:2.1}]
Define $A_i = P_i$ for $1 \leq i \leq \ell$, $A_i = Q_{i - \ell}$ for $\ell + 1 \leq i \leq 2\ell$ and similarly $B_i = Q_i$ for $1 \leq i \leq \ell$, $B_i = P_{i - \ell}$ for $\ell + 1 \leq i \leq 2\ell$.
Then $\mathcal A = \{A_1, \dots, A_{2\ell}\}$ and $\mathcal B = \{B_1, \dots, B_{2\ell}\}$ satisfy the conditions of Theorem \ref{th:2.2} with $a = b = k$.
Thus $2\ell \leq {2k\choose k}$ and thereby \eqref{eq:2.1} follows.
\end{proof}

If $\mathcal F_0 \neq \emptyset$, then one can use an extension (cf.\ \cite{F1}) of \eqref{eq:2.3} to show that \eqref{eq:2.1} is strict.

Another ingredient of the proof of Theorem \ref{th:1.6} is the following

\setcounter{theorem}{2}
\begin{theorem}[\cite{F2}]
\label{th:2.3}
Suppose that $\mathcal A \subset {[n]\choose k}$, $n > 2k \geq 6$.
Let $r$ be an integer, $4 \leq r \leq k + 1$.
If $\mathcal A$ is intersecting and $\Delta(\mathcal A) \le \Delta(\mathcal B_r)$ then
\beq
\label{eq:2.4}
|\mathcal A| \le|\mathcal B_r|.
\eeq
\end{theorem}
See \cite{KZ} for an alternative proof of this theorem. 

Let us note that if $\mathcal A$ is not a star then for all $x \in [n]$ there exists $A(x) \in \mathcal A$ with $x \notin A(x)$.
There are only ${n - 1\choose k - 1} - {n - k - 1\choose k - 1}$ sets $A \in {[n]\choose k}$ satisfying $x \in A$, $A \cap A(x) \neq \emptyset$.
Thus $|\mathcal A(x)| \leq {n - 1\choose k - 1} - {n - k - 1\choose k - 1} = |\mathcal B_{k + 1}(1)|$.
This shows that Theorem \ref{th:2.3} extends the Hilton--Milner Theorem.

The last ingredient of the proof is the Kruskal--Katona Theorem (\cite{Kr}, \cite{Ka2}).
We use it in a form proposed by Hilton \cite{H}.

For fixed $n$ and $k$ let us define the \emph{lexicographic order} $<_L$ on ${[n]\choose k}$ by setting
$$
A <_L B \ \ \ \text{ iff } \ \ \min\{x \in A \setminus B\} < \min\{x \in B \setminus A\}.
$$
For an integer $1 \leq m \leq {n\choose k}$ let $\mathcal L(m) = \mathcal L(m, n, k)$ denote the family of the first $m$ subsets $A \in {[n]\choose k}$ in the lexicographic order.

Let $a, b$ be positive integers, $a + b \leq n$.
Two families $\mathcal A \subset {[n]\choose a}$, $\mathcal B\subset {[n]\choose b}$ are called \emph{cross-intersecting} if $A \cap B \neq \emptyset$ for all $A \in \mathcal A$, $B \in \mathcal B$.

\begin{theorem}[\cite{Kr}, \cite{Ka2}, \cite{H}]
\label{th:2.4}
Let $X \subset [n]$ and $|X| \geq a + b$.
If $\mathcal A \subset {X\choose a}$ and $\mathcal B\subset {X \choose b}$ are cross-intersecting then $\mathcal L(|\mathcal A|, X, a)$ and $\mathcal L(|\mathcal B|, X, b)$ are cross-intersecting as well.
\end{theorem}

Let us sketch the proof of this for completeness. Take the family $\mathcal A^c:=\{X\in {[n]\choose n-a}: \bar X\not\in \mathcal A\}$. Consider the $b$-shadow $\partial^b(\mathcal A^c)$, consisting of all sets of size $b$ that are contained in some set from $\mathcal A^c$. Then it is easy to see that $\partial^b(\mathcal A^c)$ must be disjoint from $\mathcal B$. Since the shadow of $\mathcal A^c$ is minimized for the last $|\mathcal A^c|$ sets in the lex order (which is up to a reordering of the ground set is the same as the first $|\mathcal A^c|$ sets in the colex order), the ``best'' choice for $\mathcal A$ is the family $\mathcal L(|\mathcal A|, X, a)$. And then we naturally get that $\mathcal B$ can be taken to be $\mathcal L(|\mathcal B|, X, b)$.

Note that if $\mathcal G \subset {[n]\choose k}$ is intersecting then the two families $\mathcal G(1) \subset {[2,n]\choose k - 1}$ and $\mathcal G(\bar 1) \subset {[2,n]\choose k}$ are cross-intersecting.
Usually we apply Theorem \ref{th:2.4} to these families (with $X = [2, n]$).

In our situation with $\mathcal F \subset {[n]\choose k}$ being almost intersecting and $\mathcal F_0 \subset \mathcal F$ defined by Proposition \ref{pr:1.8}, $\mathcal F_0(1)$ and $\mathcal F(\bar 1)$ are cross-intersecting.

Using Theorem \ref{th:2.4} one easily deduces the following.

\setcounter{corollary}{4}
\begin{corollary}
\label{cor:2.5}
Let $r \geq 3$ be an integer.
Suppose that $\mathcal A \subset {[2, n]\choose k - 1}$ and $\mathcal B \subset {[2, n]\choose k}$ are cross-intersecting, $n > 2k$, $k \geq r$.
If
\beq
\label{eq:2.4masodik}
|\mathcal A| \geq {n - 1\choose k - 1} - {n - r\choose k - 1}.
\eeq
Then
\beq
\label{eq:2.5}
|\mathcal B| \leq {n - r\choose k - r + 1}.
\eeq
\end{corollary}

\begin{proof}
Note that $\mathcal L\left({n - 1\choose k - 1} - {n - r\choose k - 1}, [2, n], k - 1\right) = \left\{L \in {[2,n]\choose k - 1} : L \cap [2, r] \neq \emptyset\right\}$.
Since $n > 2k$, $[2, r] \subset B$ must hold for every $B \in {[2,n]\choose k}$ which intersects \emph{each} member of $\mathcal L\left({n - 1\choose k - 1} - {n - r\choose k - 1}, [2, n], k - 1\right)$.
Via Theorem \ref{th:2.4} this implies \eqref{eq:2.5}.
\end{proof}

\begin{corollary}
\label{cor:2.6}
Suppose that $\mathcal A \subset {[2,n]\choose k - 1}$, $\mathcal B \subset {[2, n]\choose k}$ are cross-intersecting, $n > 2k > 2$,
\beq
\label{eq:2.4harmadik}
|\mathcal B| \geq k.
\eeq
Then
\beq
\label{2.5masodik}
|\mathcal A| \leq {n - 1\choose k - 1} - {n - k\choose k - 1}.
\eeq
\end{corollary}

\begin{proof}
Just note that $\mathcal L(k, [2, n], k) = \bigl\{[2,k] \cup \{j\}, k + 1 \leq j \leq 2k\bigr\}$ and the only $(k - 1)$-sets intersecting each of these $k$-sets are those which intersect $[2,k]$.
\end{proof}

\section{Some inequalities concerning binomial coefficients}
\label{sec:3}

In this section we present some inequalities that we use in Section~\ref{sec:5}.
The proofs are via standard manipulations, the reader might just glance through them briefly.

\begin{lemma}
\label{lem:3.1}
\begin{align}
\label{eq:3.1}
{2k\choose k - 2} &\geq {2k - 1\choose k - 1} \ \ \ \text{ for } \ \ k \geq 6,\\
\label{eq:3.2}
{2k + 1\choose k - 2} &\geq {2k - 1\choose k - 1} \ \ \ \text{ for } \ \ k \geq 4.
\end{align}
\end{lemma}

\begin{proof}
${2k\choose k - 2} \bigm/ {2k - 1\choose k - 1} = \frac{2k \cdot (k - 1)}{(k + 1)(k + 2)}$ which is a monotone increasing function of~$k$.
Since for $k = 6$, $2\times 6\times 5 = 60 > 56 = 7\times 8$, \eqref{eq:3.1} is proved.
To prove \eqref{eq:3.2} just note ${2k + 1\choose k - 2} > {2k\choose k - 2}$ and check it for $k = 4$ and $5$.
\end{proof}

\begin{lemma}
\label{lem:3.2}
Suppose that $k \geq 10$ and $3k + 2 \geq m \geq 2k - 4$.
Then
\beq
\label{eq:3.3}
2 \geq {m\choose k - 2} \Bigm/{m - 1\choose k - 2} \geq 4/3.
\eeq
Moreover, if $m-s\ge 2k-4$ then
\beq
\label{eq:3.4}
\sum_{0 \leq i \leq s} {m - i\choose k - 2} \geq \left(2 - \frac1{2^s}\right) {m\choose k - 2}.
\eeq
\end{lemma}

\begin{proof}
${m\choose k - 2} \bigm/{m - 1\choose k - 2} = \frac{m}{m - k + 2}$.
Now \eqref{eq:3.3} is equivalent to
$$
2m - 2k + 4 \geq m \geq \frac43 m - \frac43 k + \frac83.
$$
The first part is equivalent to $m \geq 2k - 4$, the second to $4k - 8 \geq m$.
As for $k \geq 10$, $4k - 8 \geq 3k + 2$, we are done.
The inequality \eqref{eq:3.4} is a direct application of \eqref{eq:3.3}.
\end{proof}

\begin{lemma}
\label{lem:3.3}
Suppose that $n \geq 2(k + \sqrt{k} + 2)$, $k \geq 9$, $r \geq \sqrt{k} + 5$.
Then
\beq
\label{eq:3.5}
{n - r + 1\choose  k - r + 2} < {n - r - 1\choose k - 2}.
\eeq
\end{lemma}

\begin{proof}
Let us first show that for $n, k$ fixed the function $f(r) = {n - r + 1\choose k - r + 2}\bigm/{n - r - 1\choose k - 2}$ is monotone decreasing in $r$.
Indeed, $f{(r + 1)}/f(r) = \frac{n - r - 1}{n - r + 1} \cdot \frac{k - r + 2}{n - k - r + 1} < 1$ as both factors are less than $1$ for $n > 2k + 1$.

Consequently it is sufficient to check \eqref{eq:3.5} in the case $r = t + 1$ where $t = \left\lfloor \sqrt{k}\right\rfloor + 4$.
Fixing $k$ and thereby $r, t$, define
$$
g(n) = {n - t\choose k - t + 1} \Bigm/ {n - t - 2\choose k - 2}.
$$

\setcounter{claim}{3}
\begin{claim}
\label{cl:3.4}
For $n\in \mathbb R$ and $n \geq 2k$, $g(n)$ is a monotone decreasing function of $n$.
\end{claim}
\begin{proof} Indeed,
$$
g(n + 1) / g(n) = \frac{n - t + 1}{n - t - 1} \cdot \frac{n - k - t + 1}{n - k} \leq \frac{(n - t + 1)(n - k - 2)}{(n - t - 1)(n - k)} < 1
$$
where we used $t \geq 3$ and $ab > (a - 2)(b + 2)$ for $a > b + 2 > 0$.\end{proof}

In view of the claim it is sufficient to prove \eqref{eq:3.5} for the case $n = 2k + 2\sqrt{k} + 4$.
\beq\label{eq:3.6}
\frac{{n - t\choose k - t + 1} }{{n - t - 2\choose k - 2}} = \\
\frac{(n - t)(n - t - 1)}{(n - k - t+2)(n - k - t + 1)} \cdot \prod_{0 \leq j \leq t - 4} \frac{k - 2 - j}{n - k - 1 - j}.
\eeq

To estimate the RHS, note that the first part is at most $2 \times 2 = 4$.
As to the product part, we can use the inequality $\frac{(a - i)(a + i)}{(b - i)(b + i)} < \left(\frac{a}{b}\right)^2$, valid for all $b > a > i > 0$ to get the upper bound
$$
\left(\frac{k - \frac{t}{2}}{n - k + 1 - \frac{t}{2}}\right)^{t - 3} = \left(1 - \frac{n + 1 - 2k}{n - k + 1 - \frac{t}{2}}\right)^{t - 3}.
$$
To prove \eqref{eq:3.5} we need to show that this quantity is at most $1/4$.
We show the stronger upper bound $e^{-\frac32}$.
Using the inequality $1 - x < e^{-x}$, it is sufficient to show
$$
\frac{n + 1 - 2k}{n + 1 - k - \frac{t}{2}} > \frac3{2(t - 3)}.
$$
Plugging in $n = 2k + 2\sqrt{k} + 4$, $t = \sqrt{k} + 4$ the above inequality is equivalent to
$$
2(\sqrt{k}+1)\bigl(2\sqrt{k} + 5\bigr) > 3k + \frac{9}{2} \sqrt{k} + 9, \ \ \text{ or}
$$
$k + 9.5\sqrt{k} +1> 0$ which is true for $k\ge 0$.
\end{proof}

\setcounter{lemma}{4}
\begin{lemma}
\label{lem:3.5}
Suppose that $n \geq 3k + 3$, $k \geq 4$ then
\beq
\label{eq:3.7}
{n - 4\choose k - 3} + {2k - 1\choose k - 1} \leq {n - 5\choose k - 2} + {n - 5\choose k - 4}.
\eeq
\end{lemma}

\begin{proof}
Let us first prove \eqref{eq:3.7} in the case $n = 3k + 3$,
\beq
\label{eq:3.8}
{3k - 1\choose k - 3} + {2k - 1\choose k - 1} \leq {3k - 2\choose k - 2} + {3k - 2\choose k - 4}.
\eeq

The cases $k = 4, 5, 6$ can be checked directly.
Let $k \geq 7$.
Note that
$$
{3k - 1\choose k - 3}\Bigm/{3k - 2\choose k - 2} = \frac{(3k - 1)(k - 2)}{(2k + 1)(2k + 2)} = \frac{3k^2 - 7k + 2}{4k^2 + 6k + 2} < \frac34.
$$
Thus it is sufficient to show
\beq
\label{eq:3.9}
{2k - 1\choose k - 1}\Bigm/{3k - 2\choose k - 2} \leq \frac14.
\eeq
In view of $k \geq 7$, ${2k - 1\choose k - 1}\bigm/{2k\choose k - 2}$ is less than~$1$.
Thus \eqref{eq:3.9} will follow from
\beq
\label{eq:3.10}
{2k\choose k - 2}\Bigm/{2k + 4\choose k - 2} = \frac{(k + 6)(k + 5)(k + 4)(k + 3)}{(2k + 4)(2k + 3)(2k + 2)(2k + 1)} < \frac14.
\eeq

Since $\frac{k + i + 2}{2k + i} = \frac12 + \frac{\frac{i}{2} + 2}{2k + i}$ is a decreasing function of $k$, it is sufficient to check \eqref{eq:3.10} for $k = 7$.
Plugging in $k = 7$ we obtain $\frac{143}{612} < \frac14$, as desired.

To prove \eqref{eq:3.7} for $n > 3k + 3$, we show that passing from $n$ to $n + 1$ the RHS increases more than the LHS.
More exactly we show:
\beq
\label{eq:3.11}
{n - 4\choose k - 4} < {n - 5\choose k - 3}.
\eeq
We have
$$
{n - 4\choose k - 4} \Bigm/ {n - 5\choose k - 3} = \frac{(n - 4)(k - 3)}{(n - k)(n - k - 1)}.
$$
Using $n > 3k$, $\frac{n - 4}{n - k} < 2$ and $\frac{k - 3}{n - k - 1} < \frac12$, we get \eqref{eq:3.11}.
\end{proof}

\section{The case $k = 3$, $n \geq 13$}
\label{sec:4}
Let $\mathcal F = \mathcal F_0 \cup \mathcal P_1 \cup\ldots \cup \mathcal P_\ell$ be the canonical partition of the almost intersecting family $\mathcal F \subset {[n] \choose 3}$.
Let us make the indirect assumption that
\beq
\label{eq:4.1}
|\mathcal F| \geq |\mathcal B^+| = {n-1\choose 2}-{n-4\choose 2}+2= 3n - 7
\eeq
and that $\mathcal F$ is not isomorphic to $\mathcal B^+$. 
In view of \eqref{eq:2.1} and $2{5\choose 2} = 20<3n-6$ one has $\mathcal F_0 \neq \emptyset$. The proof at the end of  Section~\ref{sec:1} implies $\ell(\mathcal F) \geq 2$.

For notational convenience we set $(a, b, c) = \{a, b, c\}$.
By symmetry we assume $\mathcal P_1 = \{(1,2,3), (4,5,6)\}$.
Note that for $F \in (\mathcal F\setminus \mathcal P_1)$, $F \cap (1,2,3) \neq \emptyset$ and $F\cap (4,5,6) \neq \emptyset$ imply
\beq
\label{eq:4.2}
|F\setminus [6]| \leq 1
\eeq
and
\beq
\label{eq:4.3}
\{a, b\} \subset F \ \ \text{ for at least one of the $9$ choices $1 \leq a \leq 3$, $4\leq b \leq 6$.}
\eeq
For $\{a, b\}$, $1 \leq a \leq 3$, $4 \leq b \leq 6$ define $D(a, b) = \{c \in [7, n], (a, b, c) \in \mathcal F\}$.
Let $(a_1, a_2, a_3)$ and $(b_1, b_2, b_3)$ be some permutations of $(1, 2, 3)$ and $(4, 5, 6)$, respectively.

\begin{lemma}
\label{lem:4.1}

\begin{itemize}
\item[{\rm (i)}] If $D(a_i, b_i) \neq \emptyset$ for $i = 1,2,3$ then $D(a_i, b_i)$ is the same $1$-element set for $1 \leq i \leq 3$.

\item[{\rm (ii)}] If $|D(a_1, b_1)| \geq 3$ then $D(a_i, b_i) = \emptyset$ for $i = 2,3$.
\end{itemize}
\end{lemma}

\begin{proof}
Suppose by symmetry $|D(a_1, b_1)| \geq 2$ and let $x, y \in D(a_1, b_1)$.
The almost intersecting property implies $(a_i, b_i, z) \notin \mathcal F$ for $i = 2,3$ and $z \notin \{x, y\}$.
This already proves (ii).
To continue with the proof of (i) choose $x_2, x_3 \in \{x, y\}$, not necessarily distinct elements so that $(a_i, b_i, x_i) \in \mathcal F$ for $i = 2,3$.

There are two simple cases to consider.
Either $x_2 = x_3$ or $x_2 \neq x_3$.
By symmetry assume $x_3 = y$.
In the first case $(a_1, b_1, x)$ is disjoint to both $(a_2, b_2, y)$ and $(a_3, b_3, y)$.
While in the latter case $(a_3, b_3, y)$ is disjoint to both $(a_1, b_1, x)$ and $(a_2, b_2, x)$.
These contradict the almost intersecting property.
\end{proof}

\begin{lemma}
\label{lem:4.2}
If $|D(a,b)| \geq 3$ for some $1 \leq a \leq 3$, $4 \leq b \leq 6$, then $\{a, b\} \cap F \neq \emptyset$ for all $F \in \mathcal F$.
\end{lemma}

\begin{proof}
Suppose by symmetry $(a, b) = (1, 4)$ and $(1, 4, c) \in \mathcal F$ for $c = 7, 8, 9$.
Let indirectly $F \in \mathcal F$ satisfy $F \cap \{1,4\} = \emptyset$.
By \eqref{eq:4.2}, $|F \cap (7, 8, 9)| \leq 1$.
Thus $F$ is disjoint to at least two of the three triples $(1, 4, c)$, $7 \leq c \leq 9$, the desired contradiction.
\end{proof}

How many choices of $(a,b)$, $1 \leq a \leq 3$, $4 \leq b \leq 6$ can be that satisfy $|D(a, b)| \geq 3\,$?
In view of Lemma \ref{lem:4.1} (ii), $\{a, b\} \cap \{a', b'\} \neq \emptyset$ must hold for distinct choices.
Recall the easy fact that every bipartite graph without two disjoint edges is a star. Apply this on the bipartite graph with two classes $\{1,2,3\}$ and $\{4,5,6\}$ and edges corresponding to pairs $(a,b)$ with $|D(a,b)|\ge 3$ and get that all of these edges share a common vertex.
Consequently, by symmetry, we may assume that $|D(a,b)| \geq 3$ implies $a = 1$.
Let us distinguish \emph{four} cases.
$$
|D(1, j)| \geq 3 \ \ \ \text{ for } \ \ \ j = 4, 5, 6.
\leqno{\rm (a)}
$$
We claim that $\mathcal F(\bar 1) = \{(4, 5, 6)\}$.
Let us prove it.
Suppose that $F \in \mathcal F$, $1 \notin F$ and by symmetry $4 \notin F$.
Choose $(x, y, z) \subset [7, n]$ such that $(1, 4, x), (1, 4, y), (1, 4, z) \in \mathcal F$.
In view of \eqref{eq:4.2} at least two of them are disjoint to $F$, a contradiction.

Since $(1,2,3)$ is the only member of $\mathcal F$ disjoint to $(4,5,6)$, now $\mathcal F \subset \bigl\{(1, u, v): \{u, v\} \cap (4, 5, 6) \neq \emptyset\bigr\} \cup \bigl\{(1, 2, 3), (4, 5, 6)\bigr\}$ follows.
$$
|D(1, j)| \geq 3 \ \ \ \ \text{ for } \ \ \ j = 4,5, \ \ \ \text{ but } \ \ \ |D(1,6)| \leq 2.
\leqno{\rm (b)}
$$

In view of Lemma \ref{lem:4.1} (ii), $D(a,b) = \emptyset$ for $a = 2,3$ and $b=4,5,6$.
Using \eqref{eq:4.2} as well we infer
\beq
\label{eq:4.4}
\left| \mathcal F \setminus {[6]\choose 3}\right| \leq  2(n - 6) + |D(1, 6)|.
\eeq

To estimate $\left|\mathcal F \cap {[6]\choose 3}\right|$ we need another simple lemma.

\begin{lemma}
\label{lem:4.3}
If $|D(a,b)| \geq 2$ for some $1 \leq a \leq 3$, $4 \leq b \leq 6$ then $[6] \setminus \{a, b\}$ contains no member of $\mathcal F$.
\end{lemma}

\begin{proof}
If $E \in {[6]\setminus \{a, b\}\choose 3}$, then $E \cap (a, b, c) = \emptyset$ for all $c \in D(a, b)$.
Thus almost intersection implies $E \notin \mathcal F$.
\end{proof}

Applying the lemma to both $(a, b) = (1, 4)$ and $(1, 5)$ yields $\left| \mathcal F \cap {[6]\choose 3}\right| \leq 20 - 7 = 13$.

In case $|D(1, 6)| = 2$, we have
$$|\mathcal F|\le 2(n-6) +2+13 = 2n+3<3n-7 \ \ \ \ \text{ for }\ \ \ n \geq 13.$$

$$
|D(1, 4)| \geq 3 > |D(a, b)| \ \ \ \ \text{ for } \ \ \ (a, b) \neq (1,4), \ \ \ 1 \leq a \leq 3, \ \ 4 \leq b \leq 6.
\leqno{\rm (c)}
$$
In view of Lemma \ref{lem:4.1} (ii), $D(a, b) = \emptyset$ is guaranteed if $(a, b) \cap (1,4) = \emptyset$.
This leads to
\beq
\label{eq:4.5}
\left|\mathcal F \setminus {[6]\choose 3}\right| \leq n - 6 + 4 \times 2 = n + 2.
\eeq
On the other hand Lemma \ref{lem:4.3} yields
$$
\left|\mathcal F \cap {[6]\choose 3}\right| \leq 20 - 4 = 16.
$$
Together with \eqref{eq:4.5} this implies
$$
|\mathcal F| \leq n + 18 < 3n - 7 \ \ \ \ \text{ for }\ \ \ n \geq 13.
$$
$$
|D(a, b)| \leq 2 \ \ \ \ \text{ for all } \ \ \ (a, b), \ \ 1 \leq a \leq 3, \ \ 4 \leq b \leq 6.
\leqno{\rm (d)}
$$

Applying Lemma \ref{lem:4.1} (i) and (ii) gives that
$$
\bigl|D(a_1, b_1)\bigr| + \bigl|D(a_2, b_2)\bigr| + \bigl|D(a_3, b_3)\bigr| \leq 4.
$$
Using this for three disjoint matchings from the complete bipartite graph between {1,2,3} and {4,5,6} yields
$$
\left|\mathcal F \setminus {[6]\choose 3}\right| \leq 12.
$$
Thus
$$
|\mathcal F| \leq 32 \leq 3n - 7 \ \ \ \ \text{ for } \ \ \ n \geq 13.
$$
In case of equality, ${[6]\choose 3} \subset \mathcal F$.
However, that would immediately imply $\mathcal F = {[6]\choose 3}$.
Thus the proof of the case $k = 3$, $n \geq 13$ is complete.

\section{The proof of \eqref{eq:1.3} for $k \geq 4$}
\label{sec:5}

We are going to distinguish three cases according to $\Delta(\mathcal F_0)$.
$$
\Delta(\mathcal F_0) \leq {n - 2\choose k - 2} + {n - 3\choose k - 2} = {n - 1\choose k - 1} - {n - 3\choose k - 1}.
\leqno{\rm (a)}
$$
Let us suppose $n \geq 2k + 5$.
In view of \eqref{eq:3.2},
$$
{n - 4\choose k - 2} > {2k - 1\choose k - 1}.
$$
Consequently, for any choice of a full tail $\mathcal T$,
$$
\Delta(\mathcal F_0 \cup \mathcal T) \leq \! \Delta(\mathcal F_0) + \ell \leq {n - 2\choose k - 2} + {n - 3\choose k - 2} + {n - 4\choose k - 2}\! =\! {n - 1\choose k - 1} - {n - 4\choose k - 1}.
$$
Thus we may apply \eqref{eq:2.4} with $r = 4$:
\beq
\label{eq:5.1}
|\mathcal F_0 \cup \mathcal T| \leq {n - 1\choose k - 1} - {n - 4\choose k - 1} + {n - 4\choose k - 3}.
\eeq
From \eqref{eq:5.1} and $\ell \leq {2k - 1\choose k - 1}$ we infer
\beq\label{eq:5.111}
|\mathcal F| \leq {n - 1\choose k - 1} - {n - 4\choose k - 1} + {n - 4\choose k - 3} + {2k - 1\choose k - 1}.
\eeq
Using $\bigl|\mathcal B^+\bigr| > \bigl|\mathcal B_{k + 1} \bigr| \geq \bigl|\mathcal B_5\bigr|$, it is sufficient to show that the RHS is not larger than $|B_5|$.
Equivalently
\beq
\label{eq:5.2}
{n - 4\choose k - 3} + {2k - 1\choose k - 1} \leq {n - 5\choose k - 2} + {n - 5\choose k - 4}.
\eeq
Since \eqref{eq:5.2} is the same as \eqref{eq:3.7}, for $n \geq 3k + 3$ we are done.

To deal with the case (iii), we cannot be so generous. We assume that $n\le 3k+2$.
 Note that $$|\mathcal B^+|>{n-1\choose k-1}-{n-k-1\choose k-1}\ge {n-1\choose k-1}-{2k+1\choose k-1}.$$
Using \eqref{eq:5.111} and the inequality above, it is sufficient for us to show that
$${n-4\choose k-1}-{n-4\choose k-3}\ge 2{2k+1\choose k-1}.$$
The left hand side is $\big(1-\frac{(k-1)(k-2)}{(n-k-1)(n-k-2)}\big){n-4\choose k-1}\ge \big(1-\frac{k^2}{(n-k)^2}\big){n-4\choose k-1}
\ge \Big(1-\big(\frac {k}{k+2\sqrt k+4}\big)^2\Big)
{n-4\choose k-1} \ge \Big(1-\big(1-\frac {2}{\sqrt k} +\frac 1k\big)^2\Big) {n-4\choose k-1} \ge 2k^{-1/2}{n-4\choose k-1}.$
Thus, it is sufficient for us to show that
$${n-4\choose k-1}/{2k+1\choose k-1}\ge k^{1/2}.$$
Let us define $2p = n - 2k - 4$ and note $p > \sqrt{k}$. 
In view of \eqref{eq:3.3} and $n\le 3k+2$ we have
\beq
\label{eq:5.5}
{n - 4\choose k - 1} \Bigm/ {2k + 1\choose k - 1} > (4/3)^{2p -1}>p>\sqrt k,
\eeq
since $(4/3)^{2x-1}>x$ holds for all $x>0$. 
This concludes the proof of \eqref{eq:1.3} in this case.


$$
{n - 1\choose k - 1} - {n - 3\choose k - 1} < \Delta(\mathcal F_0) \leq {n - 1\choose k - 1} - {n - k\choose k - 1}.
\leqno{\rm (b)}
$$

Let $1$ be the vertex of highest degree in $\mathcal F_0$.

\begin{claim}
\label{cl:5.1}
Let $\mathcal G \subset {[n]\choose k}$ be any intersecting family containing $\mathcal F_0$.
Then $1$ is the unique vertex of highest degree in~$\mathcal G$.
\end{claim}

\begin{proof}
By assumption $|\mathcal G(1)| \geq |\mathcal F_0(1)| > {n - 2\choose k - 2} + {n - 3\choose k - 2}$.

Let $2 \leq x \leq n$ be an arbitrary vertex.
In view of Corollary \ref{cor:2.5},
$$
\bigl|\mathcal G(\bar 1, x)\bigr| \leq \bigl|\mathcal G(\bar 1)\bigr| \leq {n - 3\choose k - 2}.
$$
The inequality
$$
|\mathcal G(1, x)| \leq {n - 2\choose k - 2}
$$
is obvious.
Therefore $|\mathcal G(x)| = |\mathcal G(\bar 1, x)| + |\mathcal G(1, x)| \leq {n - 2\choose k - 2} + {n - 3\choose k - 2} < |\mathcal G(1)|$.
\end{proof}

Define the parameter $r$, $4 \leq r \leq k$ by
\beq
\label{eq:5.7}
{n - 1\choose k - 1} - {n - (r - 1)\choose k - 1} < \Delta(\mathcal F_0) \leq {n - 1\choose k - 1} - {n - r\choose k - 1}.
\eeq
Let us choose the full tail $\mathcal T$ so that $1 \notin T$ for all $T \in \mathcal T$.
Applying Claim \ref{cl:5.1} to $\mathcal G = \mathcal F_0 \cup \mathcal T$ yields $\Delta(\mathcal F_0 \cup \mathcal T) = \Delta(\mathcal F_0)$.
Thus Theorem \ref{th:2.3} implies
\beq
\label{eq:5.8}
\bigl|\mathcal F_0 \cup \mathcal T\bigr| \leq {n - 1\choose k - 1} - {n - r\choose k - 1} + {n - r\choose k - r + 1}.
\eeq

Let us first prove \eqref{eq:1.3} in the case $n \geq 3k + 3$.
Using $|\mathcal B_r| \leq |\mathcal B_k|$ and $\ell(\mathcal F) \leq {2k - 1\choose k - 1}$ it is sufficient to show
${n - 1\choose k - 1} - {n - k\choose k - 1} + {n - k\choose 1} + {2k - 1\choose k - 1} < {n - 1\choose k - 1} - {n - k - 1\choose k - 1} + 2$, or equivalently ${2k - 1\choose k - 1} < {n - k - 1\choose k - 2} - (n - k) + 2$.
For $n \geq 3 k + 3$ the RHS is an increasing function of~$n$.
Thus it is sufficient to check the case $n = 3k + 3$:
$$
{2k - 1\choose k - 1} < {2k + 2\choose k - 2} - 2k - 1 = {2k + 1\choose k - 2} + \left({2k + 1\choose k - 3} - 2k - 1\right).
$$
This inequality is true by \eqref{eq:3.2} and $k - 3 \geq 1$.

Now let us turn to the case $k \geq 10$, $3k + 2 \geq n \geq 2\bigl(k + \sqrt{k} + 2\bigr)$.
Recall the definition of $r$ from \eqref{eq:5.7}.

Using \eqref{eq:2.1} and Corollary \ref{cor:2.5} we have
\beq
\label{eq:5.9}
\ell = \ell(\mathcal F) \leq \min \left\{{2k - 1\choose k - 1}, {n - r + 1\choose k - r + 2}\right\}.
\eeq

Let us first consider the case
$$
r < \sqrt{k} + 5.
$$
We are going to prove \eqref{eq:1.3} in the form
$$
|\mathcal F| \leq {n \! -\!  1\choose k \! -\!  1} \! -\!  {n \! -\!  r\choose k \! -\!  1} + {n - r\choose k \! -\!  r + 1} + {2k \! -\!  1\choose k - 1} \leq {n \! -\!  1\choose k \! -\!  1} \! -\!  {n \! -\!  k \! -\!  1\choose k - 1},
$$
or equivalently
\begin{small}\beq
\label{eq:5.10}
{n - r\choose k - r + 1} + {2k - 1\choose k - 1} \leq {n - r - 1\choose k - 2} + {n - r - 2\choose k - 2} + \ldots + {n - k - 1\choose k - 2}.
\eeq
\end{small}
We want to apply \eqref{eq:3.4} to the RHS.
Note that $n - s \geq 2k - 4$ is satisfied if $s \leq 2 \sqrt{k} + 8$.
Since $r < \sqrt{k} + 5$, $\bigl(2 - 2^{-\sqrt{k}}\bigr){n - r - 1\choose k - 2}$ is a lower bound for the RHS.
As to ${2k - 1\choose k - 1}$, in view of \eqref{eq:3.1} and \eqref{eq:3.3} it is very small, e.g.,
$$
{2k - 1\choose k - 1} < \text{\rm RHS} \times \left(\frac43\right)^{-\sqrt{k}}.
$$
As to the main term, ${n - r\choose k - r + 1}$, using $r \geq 4$ we have
\begin{align*}
{n - r\choose k - r + 1} &\leq {n - r\choose k - 3} = {n - r - 1\choose k - 2} \frac{(n - r)(k - 2)}{(n - r - k + 3)(n - r - k + 2)} \leq\\
&\leq \frac{n - 4}{n - 4 - (k - 3)} \cdot \frac{k - 2}{n - 4 - (k - 2)} {n - r - 1\choose k - 2}.
\end{align*}
Both factors in the coefficient of ${n - r - 1\choose k - 2}$ are decreasing functions of~$n$.
Thus the maximum is attained for $n = 2k + 2\sqrt{k} + 4$ and its value is
$$
\frac{2\bigl(k + \sqrt{k}\bigr)}{\bigl(k + \sqrt{k}\bigr) + \bigl(\sqrt{k} + 3\bigr)} \cdot \frac{k - 2}{k - 2 + 2\sqrt{k} + 2} \overset{\text{\rm def}}{=} h(k).
$$

To prove \eqref{eq:5.9} it is sufficient to show
$$
h(k) + \left(\frac43\right)^{-\sqrt{k}} < 2 - 2^{-\sqrt{k}}.
$$
Since
$$
h(k) < \frac{2}{1 + \frac1{\sqrt{k}}} \cdot \frac1{1 + \frac2{\sqrt{k}}} < 2 - \frac2{\sqrt{k}},
$$
we are done.

Let us now suppose that $\sqrt{k} + 5 \leq r < k$.
We want to establish \eqref{eq:1.3} in the form
$$
|\mathcal F| = \bigl|\mathcal F_0 \cup \mathcal T\bigr| + \ell(\mathcal F) < \bigl|\mathcal B_{r + 2}\bigr|.
$$
Using \eqref{eq:5.8} and \eqref{eq:5.9} one sees that the following inequality is sufficient:
$$
{n - r\choose k - r + 1} + {n - r + 1\choose k - r + 2} \leq {n - r - 1\choose k - 2} + {n - r - 2\choose k - 2}.
$$
This inequality is the sum of \eqref{eq:3.5} applied once for $r$ and once for $r + 1$.

The final subcase is $r = k$.
Using \eqref{eq:5.8} and \eqref{eq:5.9} we obtain
$$
|\mathcal F| \leq {n - 1\choose k - 1} - {n - k\choose k - 1} + {n - k\choose 1} + {n - k + 1\choose 2}.
$$
To show $|\mathcal F| < |\mathcal B^+|$ it is sufficient to show
\beq
\label{eq:5.11}
{n - k\choose 1} + {n - k + 1\choose 2} \leq {n - k - 1\choose 3} < {n - k - 1\choose k - 2} + 2.
\eeq
The second half of \eqref{eq:5.11} is evident from $k \geq 10$ and $n > 2k + 4$.
To show the first half note that
$$
{n - k + 1\choose 1} + {n - k + 1\choose 2} = {n - k + 2\choose 2} < 2{n - k - 1\choose 2},
$$
where the last inequality is true for $n - k - 1 \geq 8$.

On the other hand, for $n - k - 1 \geq 8$ one has also $2{n - k - 1\choose 2} \leq {n - k - 1\choose 3}$, concluding the proof of \eqref{eq:5.11}. \hfill $\square$ 

$$
{n - 1\choose k - 1} - {n - k\choose k - 1} < \Delta(\mathcal F_0).
\leqno{\rm (c)}
$$
In view of Corollary \ref{cor:2.6} we have
\beq
\label{eq:5.12}
\bigl|\mathcal F_0(\bar 1)\bigr| + \ell(\mathcal F) \leq k - 1.
\eeq
On the other hand, having solved the case $\ell(\mathcal F) = 1$ in Section \ref{sec:1}, we know that $\ell(\mathcal F) \geq 2$.

The first two $k$-subsets of ${[2, n]\choose k}$ in the lexicographic order are $[2, k + 1]$ and $[2, k] \cup \{k + 2\}$.
Using Theorem \ref{th:2.4} we infer
\beq
\label{eq:5.13}
\bigl|\mathcal F_0(1)\bigr| \leq {n - 1\choose k - 1} - {n - k\choose k - 1} + {n - k - 2\choose k - 2}.
\eeq
Adding \eqref{eq:5.12}, \eqref{eq:5.13} and using $\ell(\mathcal F) \leq k - 1$ we obtain
$$
|\mathcal F| \leq {n - 1\choose k - 1} - {n - k\choose k - 1} + {n - k - 2\choose k - 2} + 2(k - 1).
$$
To prove \eqref{eq:1.3} we need
$$
{n - k - 2\choose k - 2} + 2(k - 1) < {n - k\choose k - 1} - {n - k - 1\choose k - 1} + 2.
$$
Rearranging yields
$$
2(k - 1) < {n - k - 2\choose k - 3} + 2.
$$
For $k = 4$ this is simply
$$
6 < (n - 6) + 2, \ \ \ \ \text{ i.e., } \ \ \ n \geq 11.
$$
For $k \geq 5$, $k - 3\geq 2$ and therefore
$$
{n - k - 2\choose 2} > 2(k - 2) \ \ \ \text{ is sufficient.}
$$
This inequality is satisfied for $n \geq 2k + 2$.
Indeed,
$$
{k\choose 2} = \frac{k}{2} (k - 1) > 2(k - 2) \ \ \ \text{ already for } \ k \geq 3.
$$

This concludes the entire proof. \hfill $\square$\\

{\sc Acknowledgements. } We thank the anonymous referees for carefully reading the paper and providing us with their comments. The authors acknowledge the financial support from the Ministry of Education and Science of the Russian Federation in the framework of MegaGrant no 075-15-2019-1926. The second author was partially supported  by RFBR, project number 20-31-70039 and the Council for the Support of Leading Scientific Schools of the President of the
Russian Federation (grant no. N.Sh.-2540.2020.1).

\end{document}